\documentclass[11pt]{article}

\usepackage{tikz}
\usepackage{subfigure}
\usepackage[english]{babel}
\usepackage[ruled,algosection,linesnumbered]{algorithm2e}
\usepackage[center]{caption2}
\usepackage{amsfonts,amssymb,amsmath,latexsym,amsthm}
\usepackage{multirow}
\usepackage[usenames,dvipsnames]{pstricks}
\usepackage{epsfig}
\usepackage{pst-grad} 
\usepackage{pst-plot} 
\usepackage[space]{grffile} 
\usepackage{etoolbox} 
\usepackage{indentfirst}
\makeatletter 
\patchcmd\Gread@eps{\@inputcheck#1 }{\@inputcheck"#1"\relax}{}{}

\newtheorem{thm}{Theorem}[section]
\newtheorem{cor}[thm]{Corollary}
\newtheorem{lem}[thm]{Lemma}

\newtheorem{claim}{Claim}[section]

\def\ex{\mbox{ex}}
\def\sat{\mbox{sat}}

\UseRawInputEncoding

\begin{document}

\title{Minimizing the number of matchings of fixed size in a $K_s$-saturated graph
\thanks{This work was supported by the National Natural Science Foundation of China (No. 12071453), the National Key R and D Program of China (2020YFA0713100),  Anhui Initiative in Quantum Information Technologies (AHY150200), and the Innovation Program for Quantum Science and Technology (2021ZD0302904).}
}
\author{Jiejing Feng$^a$, \quad Doudou Hei$^{b}$,\quad Xinmin Hou$^{b,c}$\\
\small $^{a}$School of Data Science\\
\small University of Science and Technology of China,  Hefei, Anhui 230026, China\\
\small $^{b}$School of Mathematical Sciences\\
\small University of Science and Technology of China,  Hefei, Anhui 230026, China\\
\small $^c$ CAS Key Laboratory of Wu Wen-Tsun Mathematics\\
\small University of Science and Technology of China, Hefei, Anhui, 230026, PR China\\
{Emails: $^{a,b}$fengjj, heidd@mail.ustc.edu.cn,\,\,\,\, $^c$ xmhou@ustc.edu.cn}
}

\date{}

\maketitle

\begin{abstract}
  For a fixed graph $F$,  a graph $G$ is said to be $F$-saturated if $G$ does not contain a subgraph isomorphic to $F$ but does contain $F$ after the addition of any new edge. 
  Let $M_k$ be a matching consisting of $k$ edges and $S_{n,k}$ be the join graph of a complete graph $K_k$ and an empty graph $\overline{K_{n-k}}$. In this paper, we prove that for $s \geq3$ and $k\geq 2$, $S_{n,s-2}$ contains the minimum number of $M_k$ among all $n$-vertex $K_s$-saturated graphs for sufficiently large $n$, and when $k \leq s-2$, it is the unique extremal graph. In addition, we also show that $S_{n,1}$ is the unique extremal graph when $k=2$ and $s=3$. 
  
\end{abstract}

\section{Introduction}
Given graphs $F$ and $G$, if $G$ does not contain $F$ as a subgraph, we say that $G$ is $F$-free. We write $\ex(n,F)$ for the Tur\'{a}n number of $F$, which is  the maximum number of edges in an $F$-free graph on $n$ vertices. This function which is considered to be a cornerstone in graph theory has been studied by many researchers (see, e.g., nice surveys such as \cite{1995What} and \cite{Z2013The}). 
  A graph $G$ is said to be $F$-saturated if $G$ does not contain $F$ as a subgraph but adding any missing edge to $G$ creates a copy of $F$. 
Given graphs $G$ and $H$, write $N(H, G)$ for the number of copies of $H$ in $G$, for a subset $S\subseteq V(G)$, write $N_G(H, S)$ for the number of copies of $H$ containing $S$ in $G$.
The generalized Tur\'an number is defined as 
$$\ex(n, H, F)=\max\{N(H, G) : G \text{ is an $n$-vertex $F$-saturated graph}\}.$$ 
This function has been systematically studied by Alon and Shikhelman~\cite{2016Many} and has received much attention, for example, in~\cite{G21DM,G21DMGT,G21arX,G22arX2,GGMV20,GP19,GP20,Hei21,doi:10.1137/19M1239052}.
As the dual of the generalized Tur\'an number, the generalized saturation number is defined as 
$$\sat(n, H, F)=\min\{ N(H,G) : G \text{ is an $n$-vertex $F$-saturated graph}\}.$$ 
When $H=K_2$, it is the classical saturation number $\sat(n, F)$.
Let $S_{n,k}$ be the join graph of a complete graph $K_k$ and an empty graph $\overline{K_{n-k}}$,  which is the graph obtained from joining all the edges between $V(K_k)$  and  $V(\overline{K_{n-k}})$.
  Erd\H{o}s, Hajnal, and Moon \cite{P1964A} studied the saturation number for complete graphs, which initiated the investigation of the graph saturation problem. They proved that 
  \begin{thm}[\cite{P1964A}]\label{THM: Erdos64}
  	For every $n \geq s \geq 2$, the saturation number 
  	\begin{align*}
  		\sat(n,K_s)=(s-2)(n-s+2)+\binom{s-2}{2}.
  	\end{align*}
 Furthermore, $S_{n,s-2}$ is the unique minimal extremal graph among $K_s$-saturated graphs on $n$ vertices.
  \end{thm}
One can see a survey for significant results of the saturation problem of graphs written by Faudree, Faudree, and Schmitt \cite{2011A} and some recent results~\cite{Lan21,Ma21}. 
Recently, Kritschgau et al~\cite{2020Few} investigated the generalized saturation function for the host graph is a clique or cycle.
Chakraborti and Loh~\cite{CHAKRABORTI2020103185} confirmed a conjecture of Kritschgau et al~\cite{2020Few} by showing the following theorem.
\begin{thm}[\cite{CHAKRABORTI2020103185}]\label{THM: Ch-Loh20}
	 For every $s>r\ge 2$, there exists a constant $n_{r,s}$ such that for all $n\ge n_{r,s}$, 
	 $$\sat(n,K_r, K_s)=(n-s+2){s-2\choose r-1}+\binom{s-2}{r}$$ Moreover, there exists a constant $c_{r,s}>0$ such that the only
	 $K_s$-saturated graph with up to $\sat(n, K_r, K_s) + c_{r,s}n$ many copies of $K_r$ is $S_{n,s-2}$.
\end{thm}
Ergemlidze et al~\cite{2021Minimizing} investigated the saturation number $\sat(n,K_{a,b}, K_s)$ and showed that  $\sat(n, K_{1,t}, K_s)=\Theta(n^{\frac{t}{2}})$, furthermore, they gave more discussion about $K_{1,2}$ and $K_{2,t}$, answering a question of Chakraborti and Loh~\cite{CHAKRABORTI2020103185}.


In this note, we continue the study of the function $\sat(n, H, K_s)$ when $H$ is a matching with $k$ edges.
  We call an $n$-vertex $F$-saturated graph (minimal) extremal graph if $N(H, G)=\sat(n, H, F)$. Let $M_k$ be a matching with $k$ edges.
The main result of this article is as follows.  
\begin{thm}\label{THM: 1}
	For every $s\geq 3$ and $k\geq 2$, there is a constant $N$ such that
	$$\sat(n, M_k, K_s)=N(M_k, S_{n,s-2})$$
	for  $n\ge N$.
	Moreover,  $S_{n,s-2}$ is the unique extremal graph  when $k \leq s-2$. 
	In addition,  $S_{n,1}$ is the unique extremal graph for $s=3$ and $k=2$.
\end{thm}

The rest of the paper is arranged as follows. Section 2 gives some preliminaries for the proof of Theorem~\ref{THM: 1}. We show Theorem~\ref{THM: 1} in Section 3. Some discussions will be given in the last section.

\section{Preliminaries}
It can be easily checked that $S_{n,s-2}$ is $K_s$-saturated. Therefore, the number of copies of $M_k$ in $S_{n,s-2}$ gives an upper bound for $\sat(n, M_k, K_s)$.
When $k>s-2$, we have 
\begin{equation}\label{EQ: k>s-2}
	0\le\sat(n,M_k,K_s)\le N(M_k, S_{n,s-2})=0.
\end{equation}
Now let us count $N(M_k, S_{n,s-2})$ when $2\le k\le s-2$. Let $A$ be an independent set of size $k$, there are $\binom{n-s+2}{k}$ many ways to pick an independent set of size $k$ in the graph  $S_{n,s-2}$. Then for an arbitrary independent set $A$, the number of $M_k$ containing $A$ is $(s-2)_k$, and each $M_k$ is counted once. In addition, there does not exist a copy of $M_k$ with more than $k$ vertices in $\overline{K_{n-s+2}}$ because  $\alpha(M_k)=k$. Therefore, we have 
\begin{lem}\label{LEM: upperbound}
Let $k, s$ be integers with $2\le k\le s-2$. Then 
\begin{align}\label{EQ: upper}
	\sat(n,M_k,K_s) &\leq N(M_k, S_{n,s-2})=\frac{(s-2)_k}{k!}n^k+O(n^{k-1}).  
\end{align}
\end{lem}
The following lemma counts the number of independent sets of fixed size in $K_s$-free graphs, which has been given in~\cite{CHAKRABORTI2020103185}.
\begin{lem}[\cite{CHAKRABORTI2020103185}]\label{LEM: 1}
	For every fixed $\ell$, there are $\Theta(n^\ell)$ independent sets of order $\ell$ in every $n$-vertex $K_s$-free graph.
\end{lem}



We give an asymptotic upper bound on the number of edges of a $K_s$-saturated graph on $n$ vertices with the minimum number of copies of $M_k$.

\begin{lem}\label{LEM: 2}
	Let  $k, s$ be integers with $2\le k\le s-2$ and let $G$ be a $K_s$-saturated graph on $n$ vertices with the minimum number of copies of $M_k$. Then for any function $f(n)\rightarrow \infty$ as $n\rightarrow \infty$, we have 
	 $|E(G)|=O(nf(n))$.
\end{lem}

\begin{proof} 
We can assume that $f(n)=O(\log n)$ and $G$ has more than $O(nf(n))$ edges. We will find a contradiction. Let $B= \{v\in V(G) : d_G(v) > f(n)\}$. Then 
$$\sum_{v\in B}d_G(v)=2|E(G)|-\sum_{v\notin B}d_G(v)\ge 2|E(G)|-(n-|B|)f(n)>O(nf(n)).$$ 
Choose a vertex $v\in B$, we select an independent set $I = \{v_1, \cdots, v_{k-1}\}$ in $V (G) \backslash \{v\}$. 
Let $V_0$ be  an empty set if $vv_1\in E(G)$, otherwise, let $V_0$ be the set such that the induced graph $G[V_0\cup \{v, v_1\}]\cong K_s$, where $V_0$ exists since $G$ is $K_s$-saturated. For every $i\in [k-2]$, denote $V_i$ the set such that the induced subgraph $G[V_i\cup \{v_i, v_{i+1}\}]\cong K_s$ after adding the edge $v_iv_{i+1}$. Let $U=V_0\cup V_1 \cup \cdots \cup V_{k-2}\cup I$ and $V'=N_G(v)-U$. 
Therefore, $|V'|\geq d_G(v)-|U|\geq d_G(v)-ks \geq \frac{1}{2} d_G(v)$ because $|U|\le ks<\frac 12f(n)<\frac 12 d_G(v)$ when $n$ is large enough.
\begin{claim}\label{CLM: cl1}
The number of $M_k$ containing $I\cup\{v\}$ is at least $\frac{1}{2} d(v)$.
\end{claim}
\begin{proof}[Proof of Claim~\ref{CLM: cl1}.] 
To prove the claim, it is sufficient to show that there is at least one copy of matching $M_k$ containing $I\cup\{v, u\}$ for every $u\in V'$.

Case 1: $vv_1\in E(G)$. Then $V_0=\emptyset$. Since $|V_i|=s-2$ and $1\le k-1<s-2$ for all $i\in[k-2]$, we can find a copy of $M_k=\{uv, w_1'v_1, w_1v_2,\ldots, w_{k-2}v_{k-1}\}$ by selecting two distinct vertices $w_1,  w_1'$ from $V_1$ and $w_i$ from $V_i-\{w_1', w_1, w_2, \ldots, w_{i-1}\}$ for $i=2, \ldots, k-2$. 
	
Case 2: $vv_1\notin E(G)$. Then $V_0\not=\emptyset$. Because $|V_i|=s-2$ and $k-1<s-2$ for all $i=0,1,\ldots,k-2$, we can find a copy of $M_k=\{uv,w_0v_1,w_1v_2,\cdots,w_{k-2}v_{k-1}\}$ by selecting $w_0$ from $V_0$ and $w_i$ from $V_i-\{w_0,w_1,\ldots, w_{i-1}\}$ for $i=1,2,\ldots,k-2$.
\end{proof}
For any $v\in B$, the subgraph $G-\{v\}$ is also $K_s$-free, so there are $\Theta(n^{k-1})$ independent sets of size $k-1$ in $G-\{v\}$ by Lemma \ref{LEM: 1}. By Claim~\ref{CLM: cl1}, the number of copies of $M_k$ containing such an independent set and $v$ is at least $\frac{1}{2}d(v)$. Therefore, we have at least $\Theta\left(\sum_{v\in B}n^{k-1}d(v)\right)= \Theta(n^kf(n))$ many copies of $M_k$ in $G$ (note that an $M_k$ copy can only be counted at most a constant time that depends on $k$), which contradicts (\ref{EQ: upper}) for sufficiently large $n$.
	
\end{proof}

The following lemma counts the number of independent sets of fixed size in graphs with bounded edges.
\begin{lem}[\cite{L1993Combinatorial}]\label{LEM: 3}
Let $\tau, \ell$ be positive integers with $\ell \leq \tau + 1$ and let $G$ be a graph on $n$ vertices with $\frac{1}{\tau} \binom{n}{2}$ many edges.  Then the number of independent sets of order $\ell$ in $G$ is at least $\binom{\tau}{\ell}(\frac{n}{\tau})^\ell$.
\end{lem}

\begin{cor}\label{cor: 1}
	For every $K_s$-saturated graph $G$ on $n$ vertices with the minimum number of copies of $M_k$, there are $(1-o(1))\binom{n}{k}$ independent sets of size $k$ in $G$, for every $k \leq s-2$.
\end{cor}

\begin{proof}
	
  Let $G$ be a $K_s$-saturated graph on $n$ vertices with the minimum number of copies of $M_k$. By Lemma~\ref{LEM: 2},   we know that the number of edges of $G$ is $O(nf(n))$. Take $\tau\rightarrow\frac n{f(n)}$ and $\ell=k$. Applying Lemma~\ref{LEM: 3} and by  a simple counting, we have that there are $(1-o(1))\binom{n}{k}$ many independent sets of size $k$ in $G$.
	
\end{proof}

\section{Proof of Theorem \ref{THM: 1}}
  Let $G$ be a $K_s$-saturated graph on $n$ vertices with $N(M_k, G)=\sat(n,M_k, K_s)$. 
  By Lemma~\ref{LEM: 2}, we may assume $e(G)=O(n\log n)$.
  Let $A = \{v \in V(G) : d_G(v) \leq n^{\frac{1}{3}}\}$, then $|A| = n-o(n)$.
  
\begin{claim}\label{c3.1}
  	There are $(1-o(1))\binom{n}{k}$ independent sets of size $k$ of which all vertices are from $A$ and any two of them do not have common neighbors in $A$.
  \end{claim}
 
 \begin{proof}
  Let $\mathcal{I}$ denote the collection of such independent sets. First, there are $o(n^k)$ independent sets of size $k$ each containing at least one vertex from $V \backslash A$, since almost all vertices are in $A$. The number of independent sets of size $k$ in which all vertices are from $A$ and there exist two vertices with common neighbors in $A$ is at most $|A|\binom{n^{\frac{1}{3}}}{2}n^{k-2}=o(n^k)$. From Corollary \ref{cor: 1} and the definition of $\mathcal{I}$, we know that there are $(1-o(1))\binom{n}{k}$ independent sets in $\mathcal I$.
 \end{proof} 
  
  
  Consider an arbitrary independent set $I = \{v_1,\cdots, v_k\} \in \mathcal{I}$, there is a set $V_{i,j} \subseteq V \backslash A$ of $s-2$ vertices such that the induced graph $V_{i,j}\cup \{v_i, v_j\}$ is isomorphic to $K_s$ after adding the edge $v_iv_j$ for every $i, j \in \binom{[k]}{2}$.
   For every independent set $I\in\mathcal{I}$, let $N'_G(M_k, I)$ denote the number of copies of $M_k$ containing $I$ and $k$ vertices from $V \backslash A$. Then 
   $$N'_G(M_k, I)\ge (s-2)_k,$$
equality holds if and only if all $V_{i,j}$'s  are the same and {there is no common neighbor of $v_i$ and $v_j$ in $(V \backslash A) \backslash V_{ij}$.} Therefore, 
 $$N(M_k, G)\ge |\mathcal{I}|\cdot N'_G(M_k, I)\ge (1-o(1)){n\choose k}(s-2)_k.$$
Therefore, the upper bound \eqref{EQ: upper} is asymptotically tight for all $k$, moreover, we conclude that the number of independent sets $I = \{v_1,\cdots, v_k\} \in \mathcal{I}$ for which $V_{i,j}$'s are the same and there is no common neighbor of $v_i$ and $v_j$ in $(V \backslash A) \backslash V_{ij}$ is $(1-o(1))\binom{n}{k}$, since, otherwise, there will be more copies of $M_k$ in $G$ than the upper bound \eqref{EQ: upper}, which is a contradiction. Denote this collection of such independent sets by $\mathcal{J}$. Then $|\mathcal{J}|=(1-o(1))\binom{n}{k}$.
  
\begin{claim}\label{LEM: 4}
For $k \geq 3$, $G$ contains a subgraph isomorphic to $S_{n-o(n), s-2}$.
	
\end{claim}

\begin{proof}
Since $|\mathcal J| = \binom{n}{k}-o(n^k)$, there must exist two vertices $u,v$ in $A$ such that there are $\binom{n}{k-2}-o(n^{k-2})$ independent sets in $\mathcal{J}$ containing both $u$ and $v$. We denote the collection of independent sets in $\mathcal{J}$ containing both $u$ and $v$ by $\mathcal{K}$. Moreover, there is a set $T \subseteq V \backslash A$ of $s-2$ vertices such that the induced graph $T\cup \{u, v\}$ is isomorphic to $K_s$ after adding the edge $uv$. By the definition of $\mathcal{J}$, we can conclude that $T$ is the only common $(s-2)$-neighborhood of every independent set in $\mathcal{K}$. Therefore, $\cup_{I\in\mathcal{K}}I$ is an independent set since $G$ does not have a copy of $K_s$. For $k \geq 3$, since $\mathcal{K}=\binom{n}{k-2}-o(n^{k-2})$, we  can check that $|\cup_{I\in\mathcal{K}}I|=n-o(n)$ easily. (Note that it is not true for $k = 2$.) 
	
\end{proof}

Now we are ready to complete the proof of Theorem~\ref{THM: 1}. 	
 First,  let us consider the case when $k\geq3$.	
 Choose a maximum subgraph $S_{n_1, s-2}$ in $G$. By  Claim~\ref{LEM: 4}, we know that $n_1=n-o(n)$. 
 Denote the set of all vertices outside of $V(S_{n_1, s-2})$ by $M$. Let $m=|M|+s-2$. Then $m=n-n_1+s-2= o(n)$.	We may assume that $M\not=\emptyset$ for the sake of contradiction. Let $I$ be the independent set of $S_{n_1, s-2}$ and $K=V(S_{n_1,s-2})\setminus I$. Then $|I|=n_1-s+2=n-m$ and $|K|=s-2$. 
  
	\begin{claim}\label{c1}
		For any $v\in I$, $v$ has at least one neighbor in $M$.
	\end{claim}

\begin{proof}
Assume that there is one vertex $v\in I$ that has no neighbor in $M$. Then for any $u\in M$, there is a set $U\subseteq V \backslash I$ of $s-2$ vertices such that the subgraph induced by $V\cup \{u,v\}$ is isomorphic to $K_s$ in $G+uv$. Clearly, $U\cap (I\cup M)=\emptyset$. Therefore,  $U$ must be $K$. This fact forces $u$ to have  no  neighbor in $I$; otherwise,  we have a contradiction to  $G$ being  $K_s$-free. Thus $u$ can be added in $I$ to obtain a larger subgraph $S_{n_1+1, s-2}$, a contradiction. The claim holds.
\end{proof}

  Let $z$ be the integer such that $(z)_k = \lceil\sqrt{m + k}\cdot n^{k-\frac{1}{2}}\rceil$.
	
	\begin{claim}\label{c2}
	When $n$ is sufficiently large, for any $v\in M$, $v$ has at most $z$ neighbors in $I$.
	\end{claim}

\begin{proof}
  Suppose there is a vertex $v\in M$ that has more than $z = o(n)$ neighbors in $I$. 
  We know that there are at least $(s-2)_k\cdot \binom{n-m}{k}$ many copies of $M_k$ in $S_{n_1,s-2}$. 
  Now arbitrarily choose a $k$-set from $N_G(v)\cap I$ and a $(k-1)$-set from $K$, then, together with $v$, we have at least one copy of $M_k$ containing these vertices. The number of such copies of $M_k$ is at least $\binom{z}{k}$.
  Therefore, for large enough $n$, 
	\begin{align*}
		N(M_k, G) &\geq (s-2)_k\cdot \binom{n-m}{k}+\binom{z}{k}\\
		&\geq \frac{(s-2)_k}{k!}\cdot (n-m-k)^k+\frac{(z)_k}{k!} \\
		&\geq\frac{(s-2)_k}{k!}\cdot n^k-\frac{(s-2)_k}{k!}\cdot k(m+k)n^{k-1}+\frac{\sqrt{m+k}n^{k-\frac{1}{2}}}{k!}\\
		&\geq\frac{(s-2)_k}{k!}\cdot n^k+\frac{\sqrt{m+k}n^{k-\frac{1}{2}}}{2k!}\\
		&\geq\frac{(s-2)_k}{k!}\cdot n^k+\Theta(n^{k-\frac{1}{2}}),
	\end{align*}
which is a contradiction to the upper bound~\eqref{EQ: upper}.
\end{proof}
	
  Now we return to the analysis of $M$. 
By Claim~\ref{c1},  $|E_G(M, I)|\ge n-m$. By Claim~\ref{c2}, $|E_G(M, I)|\le z(m-s+2)$. If $m \leq \log n$, 
then, for large enough $n$, $z(m-s+2)<n-m$, a contradiction. 
Now we assume $m >\log n$. For a vertex $v\in M$, select a $(k-1)$-set $B$ from $I-N_G(v)$ (such a $B$ exists guaranteed by Claim~\ref{c2}). 
Denote $B=\{v_1, v_2, \ldots , v_{k-1}\}$ and let $V_i$ be a $(s-2)$-set of $V(G)$ such that  $V_i\cup \{v,v_i\}$ induces a subgraph isomorphic to $K_s$ after adding the edge $vv_i$. Because of the maximality of $S_{n_1, s-2}$, $K\nsubseteq N_G(v)$  and thus $V_i\not=K$ for any $i \in [k-1]$. 
Let $$\mathcal{M}(v,B)=\{M_k : M_k=\{vv', v_1v_1', \ldots , v_{k-1}v_{k-1}'\} \text{ with $v'\in V_1$ and $v_i'\in K$ for $i\in[k-1]$}\}.$$ 
It is easy to check that $|\mathcal{M}(v,B)|\ge (s-2)_k + 1$. 
Clearly, $\mathcal{M}(v,B)\cap \mathcal{M}(v,B')=\emptyset$ if $B$ and $B'$ are different $(k-1)$-sets of $I-N_G(v)$,  and $\mathcal{M}(v,B)\cap \mathcal{M}(u,C)=\emptyset$ for $u\not=v$ in $M$, where $B$ and $C$ are $(k-1)$-sets of $I-N_G(v)$ and $I-N_G(u)$, respectively.
Thus,  the total number of such copies of $M_k$ is 
$$\sum_{v\in M, B\subseteq I-N_G(v)\atop |B|=k-1}|\mathcal{M}(v,B)|\ge\binom{n-m-z}{k-1}(m -s + 2) ((s-2)_k+ 1).$$ 
Therefore, 
	\begin{align*}
	N(M_k, G) &\geq (s-2)_k\cdot \binom{n-m}{k}+\binom{n-m-z}{k-1}(m-s + 2)((s-2)_k + 1)\\
		&\geq \frac{(s-2)_k}{k!}\cdot (n-m-k)^k+\binom{n-m-z}{k-1}(m -s)((s-2)_k + 1) \\
		&\geq\frac{(s-2)_k}{k!} n^k-\frac{(s-2)_k}{(k-1)!}\cdot (m+k)n^{k-1}+\frac{(s-2)_k}{(k-1)!}(m-s)(n-m-z-k)_{k-1}\\
		&\quad+\binom{n-m-z}{k-1}(m-s)\\
		&\geq\frac{(s-2)_k}{k!} n^k-\frac{(s-2)_k}{(k-1)!}\cdot mn^{k-1}-\Theta(n^{k-1})+\frac{(s-2)_k}{(k-1)!}\cdot mn^{k-1}+\Theta(mn^{k-1})\\
		&\quad-\Theta(m^2n^{k-2})-\Theta(mzn^{k-2})\\
		&\geq\frac{(s-2)_k}{k!} n^k+\Theta(mn^{k-1}),
	\end{align*}
the last inequality holds when $n$ is sufficiently large, but this is a contradiction to the upper bound~\eqref{EQ: upper} (since we assume $m>\log n$ in this case).   Therefore, $M$ has to be an empty set, and $G\cong S_{n,s-2}$. This completes the proof of the case $k \geq 3$. 

Next we assume $k=2$ and divide the proof into two cases. 

\noindent{\bf Case 1:} $s=3$.

Then, by~\eqref{EQ: k>s-2},  $\sat(n,M_2,K_3)=0$ and $S_{n,1}$ is an extremal graph.
We claim that $S_{n,1}$ is the unique extremal graph. 
Since $G$ is $K_3$-saturated and contains the minimum number of copies of $M_2$, $G$ is connected and contains no cycle, which implies that $G$ is a tree. 
Furthermore, adding any missing edge to $G$ creates a copy of $K_3$ forces that $G\cong S_{n,1}$. 
Therefore, $S_{n,1}$ is the unique extremal graph when $k=2$ and $s=3$.  

\noindent{\bf Case 2:} $s \geq 4$.

Since $G$ is $K_s$-saturated, $\delta(G)\ge s-2$. Suppose $m=|E(G)|$.  
Then $$N(M_2, G)=\frac{1}{2}\sum_{uv \in E(G)}\left(|E(G)|-d_G(v)-d_G(u)+1\right)=\frac{1}{2}\left(m^2+m-\sum_{v \in V(G)}d_G^2(v)\right).$$
If there are two vertices $u, v\in V(G)$ such that $s-2<d_G(u)\le d_G(v)<n-1$.
We construct a new graph $G^*$ with $d_{G^*}(v)=d_G(v)+1$, $d_{G^*}(u)=d_G(u)-1$, and $d_{G^*}(x)=d_G(x)$ for $x\in V(G) \backslash \{u,v\}$ (this can be easily done just by deleting a neighbor  of $v$ and adding it to the neighborhood of $v$).  Then 
	\begin{align*}
	N(M_2, G)-N(M_2, G^*)&=\frac{1}{2}(d_{G^*}^2(v)+d_{G^*}^2(u)-d_G^2(v)-d_G^2(u))\\
	                     &=\frac{1}{2}((d_G(v)+1)^2+(d_G(u)-1)^2-d_G^2(v)-d_G^2(u))\\
	&=d_G(v)-d_G(u)+1>0, 
	\end{align*}
Continue  the above process until there is at most one vertex $v$ with  $d_G(v)\notin\{s-2,n-1\}$. For convenience, we still use $G$ to denote the  graph at the final step.

\noindent{\bf Subcase 2.1:}  $d_G(v)\in \{s-2,n-1\}$ for every $v\in V(G)$.

Assume there are $a$ vertices with degree $s-2$ and $b$  vertices with degree $n-1$. Then  we have the following equation array,
	\[
	\begin{cases}
		2m=a(s-2)+b(n-1)\\
	    n=a+b	
	\end{cases}.\]  
Solve the equation array, we have $a=\frac{n^2-n-2m}{n-s+1}$, $b=\frac{2m+2n-ns}{n-s+1}$. 
Therefore,
	\begin{align}\label{EQ: no.M_2}
		N(M_2, G)&=\frac{1}{2}(m^2+m-a(s-2)^2-b(n-1)^2)\nonumber\\
		&=\frac{1}{2}\left(m^2+(7-2n-2s)m+n(n-1)(s-2)\right)\\
		&\ge N(M_2, S_{n,s-2}),\label{EQ: eq}
	\end{align}
the last inequality holds because the quadratic function (\ref{EQ: no.M_2}) increases in the interval $[n+s-\frac{7}{2}, +\infty)$ 
and $m=|E(G)|\ge |E(S_{n,s-2})|=(s-2)(n-s+2)+\binom{s-2}{2}\geq n+s-\frac{7}{2}$  by Theorem~\ref{THM: Erdos64}, moreover,  equality holds in (\ref{EQ: eq}) if and only if $G\cong S_{n,s-2}$. 
	
\noindent{\bf Subcase 2.2:} There exists one vertex $v$ with $s-2<d_G(v)=c<n-1$.

Assume there are $a$ vertices with degree $s-2$ and $b$  vertices with degree $n-1$, then  
	\[
	\begin{cases}
		2m=a(s-2)+b(n-1)+c\\
		n=a+b+1.	
	\end{cases}\] 
Solve the equation array, we have $a=\frac{(n-1)^2-2m+c}{n-s+1}$ and $b=\frac{2m-c-(n-1)(s-2)}{n-s+1}$.
	\begin{align}
	N(M_2, G)&=\frac{1}{2}(m^2+m-a(s-2)^2-b(n-1)^2-c^2)\nonumber\\
		&=\frac{1}{2}\left(-c^2+(n+s-3)c+m^2+(7-2n-2s)m+(n-1)^2(s-2)\right)\nonumber\\
		&>\frac{1}{2}\left(m^2+(7-2n-2s)m+n(n-1)(s-2)\right)\label{EQ: >}\\ 
		&\ge N(M_2, S_{n,s-2})\nonumber
	\end{align}
where the inequality is tight in (\ref{EQ: >}) because 	$\min\limits_{s-2<c<n-1}\{-c^2+(n+s-3)c\}>(n-1)(s-2)$.
The proof of Theorem~\ref{THM: 1} is completed.     \qed

\section*{Author contributions}
{\bf Jiejing Feng:} Conceptualization, Methodology, Writing. {\bf Doudou Hei:} Conceptualization, Methodology, Writing. {\bf Xinmin Hou:} Supervision.

\section*{Statements and Declarations}
No conflict of interest exits in the submission of this manuscript, and manuscript is approved by all authors for publication. I would like to declare on behalf of my co-authors that the work described was original research that has not been published previously, and not under consideration for publication elsewhere, in whole or in part. All the authors listed have approved the manuscript that is enclosed.

\end{document}